\documentclass[11pt]{amsart}

\usepackage{amsmath}
\usepackage{amsfonts}
\usepackage{amssymb}
\usepackage{amsthm}
\usepackage{latexsym}
\usepackage[english]{babel}
\usepackage{hyperref}

\def \B {\mathbb{B}}

\def \BB {\mathfrak{B}}

\def \cc {{\rm {cap}}}

\def \dd {{\rm{d}}}
\def \dist {{\rm{dist}}}
\def \de {\partial}

\def \e {\varepsilon }

\def \fe {for every }

\def \loc {{\rm {loc}}}
\def \LIPO {{\rm {Lip}}_0}

\def \L {\mathcal{L}}
\def \LL {\mathfrak{L}}

\def \O {\Omega}

\def \phi {\varphi }

\def \RN {\mathbb{R}^N}

\def \R {\mathbb{R}}
\def \ss {\subseteq}
\def \st {such that }

\def \tes {there exists }

\def \te {there exist }

\def \W {W^1(\O,X)}
\def \WO {W_0^1(\O,X)}
\def \WL {W_{\loc}^1(\Omega,X)}
\DeclareMathOperator{\bint}{\int\hspace{-12pt}---}

\newtheorem{theorem}{Theorem}[section]
\newtheorem{lemma}[theorem]{Lemma}
\newtheorem{proposition}[theorem]{Proposition}
\newtheorem{corollary}[theorem]{Corollary}

\newtheorem{remark}[theorem]{Remark}

\theoremstyle{definition}
\newtheorem{definition}[theorem]{Definition}

\numberwithin{equation}{section}

\begin{document}

\title
{Wiener criterion for X-elliptic operators
}

\author[G. Tralli, F. Uguzzoni]{ $ $ Giulio Tralli, Francesco Uguzzoni}

\address{Dipartimento di Matematica,
         Universit\`{a} degli Studi di Bologna,
         Piazza di Porta S. Donato 5, Bologna 40126, Italy.
         }

\curraddr{}
	\email{giulio.tralli2@unibo.it}
  \email{francesco.uguzzoni@unibo.it}

\date{}

\begin{abstract}
In this note we prove a Wiener criterion of regularity of boundary points for the Dirichlet problem related to $X$-elliptic operators in divergence form enjoying the doubling condition and the Poincar\'e inequality. As a step towards this result, we exhibit some other characterizations of regularity in terms of the capacitary potentials. Finally, we also show that a cone-type criterion holds true in our setting.

\end{abstract}

\maketitle

\let\thefootnote\relax\footnotetext{{\it 2010 Mathematics Subject Classification.} 35J70, 35H20, 31C15, 31E05.}
\let\thefootnote\relax\footnotetext{{\it Keywords}: degenerate-elliptic equations, boundary regularity, Wiener criterion.}
%
\section{Introduction}
\label{intro}
 Let us consider an $X$-elliptic operator (in the sense of Lanconelli and Kogoj \cite{LK}) with measurable coefficients, in the divergence form
  \begin{equation}\label{nomeoper}
 \textstyle
 \L u=\sum_{i,j=1}^N    \de_i (b_{i,j}(x)\de_j u).
 \end{equation}
The operator $\L$ is degenerate-elliptic but its degeneracy is controlled by a family $X$ of vector fields with suitable properties. More explicitly, we assume that the coefficients of the matrix $B(x)=(b_{i,j}(x))=(b_{j,i}(x))$ satisfy the so called $X$-ellipticity condition
    \begin{equation}\label{xell}
    \textstyle
    \lambda\sum_{j=1}^m \langle X_j(x),\xi\rangle^2
    \le\langle B(x)\xi,\xi\rangle
    \le \Lambda\sum_{j=1}^m \langle X_j(x),\xi\rangle^2,
 \end{equation}
 for every $\xi\in\RN$ and for every $x$ in a bounded open set $O\subseteq\RN$. Here $X=(X_1,\ldots,X_m)$ is a family of locally Lipschitz vector fields in $\RN$ with a well-defined control distance $d$ which is continuous w.r.t.\ the Euclidean topology. We suppose that the following local doubling condition holds for the $d$-metric balls $B_s(x)$: for every compact set $K$ of $\RN$, \te $A>1$ and $R_0>0$ \st
 \begin{equation}\label{DC}
 0<|B_{2r}|\le A|B_r|,
 \end{equation}
\fe $d$-ball $B_r$ centered at a point of $K$ and of radius $r\le R_0$. Hereafter we denote by $|E|$ the Lebesgue measure of $E$. We also assume the following Poincar\'e inequality: for every compact set $K$ of $\RN$ \tes a positive constant $C$ \st
 \begin{equation}\label{PI}
  \bint_{\! B_r}|u-u_r|\,
  \le C r\, \bint_{\! B_{2r}}|Xu|,
 \end{equation}
 \fe $C^1$ function $u$ and \fe $d$-ball $B_r$ centered at a point of $K$ and of radius $r\le R_0$. We have used the notations
 \begin{equation*}
  u_r= \bint_{B_r}u
  =\frac{1}{|B_r|}\int_{B_r}u,\quad\mbox{and}\quad \left|Xu\right|^2=\sum_{j=1}^m (X_ju)^2.
 \end{equation*}
 
We remark that, for example, the PDOs of the form
    \begin{equation*}
 \textstyle
 \sum_{i,j=1}^m    X_i^* (a_{i,j}(x)X_j u)
 \qquad\text{and}\qquad
 \sum_{i,j=1}^m    X_i (a_{i,j}(x)X_j u),
 \end{equation*}
 where $(a_{i,j})$ is an $m\times m$ uniformly elliptic symmetric matrix of measurable functions, can be written (up to l.o.t.) as $X$-elliptic operators in the form (\ref{nomeoper})-(\ref{xell}).
 Moreover any second order linear PDO with nonnegative characteristic form and sufficiently smooth coefficients is $X$-elliptic w.r.t.\ a suitable family $X$ of vector fields. Indeed, if the matrix $(a_{i,j})$ is nonnegative definite and $a_{i,j}\in C^2$, then, by a result of Phillips and Sarason \cite{PS}, \tes a nonnegative definite locally Lipschitz matrix $(\alpha_{i,j})$ \st
  $\sum_{i,j=1}^N    a_{i,j}(x)\xi_i\xi_j =\sum_{h=1}^N (\sum_{j=1}^N \alpha_{h,j} \xi_j )^2$ \fe $\xi\in\RN$. We refer to \cite{LK, GL} for more details and comments.
  \\
  \indent
  The notion of $X$-ellipticity was explicitly introduced by Lanconelli and Kogoj in \cite{LK} where a Harnack inequality was proved for the equation $\L u=0$, but the same ideas were already used, for the first time in a non-euclidean context, by Franchi and Lanconelli in \cite{FL1,FL2,FL3}. Several authors have enlightened the fundamental role of conditions \eqref{DC} and \eqref{PI} in the study of PDEs modeled on vector fields, see e.g.\ the survey in \cite{HK} and references therein. See also the recent papers \cite{K, KLd, KS} for some other examples of $X$-elliptic operators (not in the H\"ormander class) enjoying \eqref{DC}-\eqref{PI}.
  Guti\'errez and Lanconelli \cite{GL} established maximum principles and homogeneous Harnack inequalities for $X$-elliptic operators with lower order terms and, in the case of dilation invariant vector fields $X$, nonhomogeneous Harnack inequalities and Liouville theorems. Other Liouville theorems are also proved in \cite{KL}. Some estimates of the Green function for the $X$-elliptic operator $\L$ were proved in \cite{Maz} in the special case that the measure of the $d$-balls behaves like a power $r^Q$ of the radius $r$ and, only recently, in \cite{U} without this restriction. A nonhomogeneous Harnack inequality is also proved in \cite{U}. We also refer to \cite{BBLU, CX, DGL, J, LU, ZF} for other related papers.

Our aim is to prove a Wiener criterion of regularity of boundary points for the Dirichlet problem related to the X-elliptic operator $\L$ in \eqref{nomeoper}. In order to state our results, we need to go into some more details. Let us fix a compact set $K_0$ of $\RN$ with interior containing the closure of $O$ and set $Q=\log_2 A$, where $A$ is the doubling constant in \eqref{DC} which can be assumed \st  $Q>2$. We recall that our assumptions on the vector fields $X_j$ imply the Sobolev inequality
\begin{equation}\label{Sob}
 \| u\|_{\frac{2Q}{Q-2}}
 \le C(D)\, \|Xu\|_2
 \qquad\text{\fe } u\in C_0^1(D)
 \end{equation}
on every open set $D$ with diameter small enough and with closure contained in the interior of $K_0$ (see e.g.\,\cite{HK}). Let now $D\subseteq O$ be a bounded domain supporting the above Sobolev inequality. We shall assume the following condition: 
\begin{equation}\label{uconn}
\text{\emph{the boundaries of the small $d$-balls contained in $D$ are connected}}. 
\end{equation}
This ensures the validity of a crucial Harnack inequality on $d$-rings, which is exploited in \cite{U} to prove two-sided pointwise estimates for the Green function of $\L$. We would like here to point out some more comments on condition \eqref{uconn}. In the (elliptic) case of $N$ vector fields $X_1,\ldots,X_N$ linearly independent at any point, it is known that the small spheres of the control metric are homeomorphic to the euclidean ones and so are connected. But, even in this case, the same property may fail to be true if the sphere is not small. We can convince ourselves of this fact by taking in $\R^2$ vector fields in the form $X_1=\phi(x)\de_{x_1}$, $X_2=\phi(x)\de_{x_2}$, with $0<\delta\leq\phi\leq 1$, $\phi(x_0)=\delta$, and $\phi\equiv 1$ outside a neighborhood of $x_0$, so that the $d$-balls centered far enough from $x_0$ have disconnected boundaries for some radii. In the sub-Riemannian case, it is proved in \cite{B} that small $d$-spheres are homeomorphic to euclidean spheres (and so are connected) if $X$ is a step $2$ distribution of vector fields, or if our vector fields are invariant under some group of dilations (see also \cite{M}).
  
Under the hypotheses we have just fixed, we shall prove (see Theorem \ref{thwin} below) that the $\L$-regularity of a boundary point $y$ of an open set $\O$ compactly contained in $D$ is related to the behavior (near $\rho=0$) of the integral  
$$\int_0^{\delta_y}\cc\left(\overline{B_\rho(y)}\smallsetminus\O\right)\frac{\rho}{\left|B_\rho(y)\right|}\,d\rho.$$ 
We refer to the beginning of Section \ref{cp3} for the definition of the $\L$-capacity $\cc$. By using the doubling property, it is easy to recognize that the behavior of this integral is equivalent to the one of the series $$\sum_k\frac{\lambda^{2k}\cc(\overline{B_{\lambda^k}(y)}\smallsetminus\O)}{\left|B_{\lambda^k}(y)\right|},$$ 
for $0<\lambda<1$ (see \cite{N}).\\
As in the classical elliptic case (see Littman, Stampacchia, and Weinberger \cite{LSW}), our criterion allows us to deduce that the $\L$-regularity of a boundary point of $\O$ does not actually depend on the coefficients $b_{ij}$ of the operator $\L$ but just on the vector fields $X_j$'s (see Corollary \ref{corind}). 

We recall that Wiener tests of regularity were proved in \cite{LSW, GW} for classical elliptic equations with measurable coefficients, in \cite{HH, NS, N} for H\"ormander operators, in \cite{FJK, CW} for degenerate elliptic equations with weights, in \cite{Bir, BM} for Poincar\'e-Dirichlet forms. The approach we follow in the proof of our Wiener criterion is inspired by the papers \cite{LSW, FJK}. In particular, in the position of the problem we adapt the classical formulation given in \cite{LSW}. In our exposition we try to enlight what are the tools really needed in order to get the result. Indeed, we do not use quasicontinuity arguments nor multiple characterizations of capacity as done in \cite{LSW, FJK}.

Our starting point are the estimates of the Green function proved in \cite{U}. We explicitly remark that in \cite{U} the following further hypothesis on $D$ was assumed: there exist $r_1,\,\theta>0$ such that 
 \begin{equation}\label{AI}
 |B_r(x)\setminus D|\ge\theta|B_r(x)|\qquad\mbox{for every }r\in]0,r_1]\mbox{ and } x\in\de D.
 \end{equation}
Moreover, it was also assumed that the conditions \eqref{DC}, \eqref{PI}, and \eqref{uconn} hold true for any $d$-ball contained in $D$. Here these further hypotheses are not necessary because of the local nature of the notion of regularity (provided by Lemma \ref{indue} below) and the fact that any domain can be approximated by domains satisfying condition \eqref{AI} (see \cite[Lemma 3.7]{U}).

  \indent
  This note is organized as follows. 
  In Section \ref{randb} we introduce the notions of barrier and of regularity of boundary points. To this aim we first prove and exploit a Caccioppoli-type estimate. In Section \ref{cp3} we study the relationship between regularity and capacitary potentials, and we prove some characterizations of the regularity. In Section \ref{www} we conclude the proof of our Wiener test. Moreover we also provide a cone criterion which ensures the regularity at a boundary point where condition \eqref{AI} is satisfied.

\section{Regularity and barriers}\label{randb}

For any open set $\O\subseteq D$, we define $W_0^1(\O,X)$ as the closure of $C_0^1(\O)$ w.r.t.\ the norm $\|Xu\|_2$, whereas $W^1(\O,X)=\{u\in L^2(\O)\,|\, Xu\in L^2(\O)\}$ is equipped with the norm $\|u\|_2+\|Xu\|_2$.

By using some results of \emph{good approximation} for functions in $W^1(\O,X)$ (see e.g.\,\cite{FSSC}), one can prove that many general properties about $W^1$-functions hold true also in our setting. For instance, the following facts will be used several times throughout the paper without further comments. Any function in $W^1(\O,X)$ vanishing in a neighborhood of $\de\O$ belongs to $W_0^1(\O,X)$. If $f\in W^1(\O,X)$ and $0\leq f\leq g$ a.e in $\O$ for some $g\in W_0^1(\O,X)$, then $f\in W_0^1(\O,X)$. If $f\in W^1(\O,X)$ is continuous in a neighborhood of $\de\O$ and $f=0$ on $\de\O$, then $f\in W_0^1(\O,X)$.

Let us recall the definition of solution to the equation $\L u=0$. To this aim, let us consider the bilinear form
  \begin{equation}\label{Luv}
 \LL(u,v)=\int_\O \langle B(x)\nabla u,\nabla v\rangle\, dx
 \end{equation}
for $u\in C^1(\O)$ and $v\in C^1_0(\O)$. By using the uniform $X$-ellipticity of $\L$ and the Sobolev inequality \eqref{Sob}, $\LL$ can be extended continuously to $\W \times\WO$. We shall say that a function $u\in\W$ is a (weak) solution to
 $\L u=0$ in $\O$, if $\LL(u,v)=0$ \fe $v\in\WO$. A function $u\in \WL$ will be called a (weak) solution to the same equation in $\Omega$ if it is a weak solution in every domain with closure contained in $\O$. In \cite[Proposition 2.4]{GL}, the authors showed that, for any $h\in\W$, there exists a unique function $u\in\W$ solution of $\L u=0$ in $\O$ such that $u-h\in\WO$. We note that the application $h\mapsto u$ clearly factors through the quotient, i.e. we have $\hat{B}: \frac{\W}{\WO}\rightarrow \W$.\\
For $u\in\W$ and $l\in\R$, we will also say that $u\leq l$ on $\de\O$ if 
$$(u-l)^+=\max\{u-l,0\}\in\WO.$$
We denote $\sup_{\de\O}u=\inf\{l\in\R\,:\,u\leq l\mbox{ on }\de\O\}$. 
In \cite[Theorem 3.1]{GL}, the following maximum principle is proved: if $u\in\W$ is a weak solution of $L u=0$ in $\O$, then we have the inequality
\begin{equation}\label{MP}
 \sup_\O u^+\leq \sup_{\de\O}u^+.
 \end{equation}
By definition, it is not difficult to show that $\sup_{\de\O}\left|h\right|$ defines actually a norm in
$$H=\frac{\{h\in\W\,:\,\sup_{\de\O}\left|h\right|<+\infty\}}{\WO},$$
which we will denote by $\left\|\cdot\right\|_H$. Thus, the maximum principle ensures the boundedness of the map
$$\left.\begin{array}{cccc}
\tilde{B}: & H & \longrightarrow & L^{\infty}(\O)  \\
\, & h &	\mapsto & u=\hat{B}h.
\end{array}\right.$$
To introduce the notion of regularity, we need to associate a solution of $\L u=0$ in $\O$ to any function in $C(\de\O)$. We are going to follow the lines of the procedure of the celebrated paper \cite{LSW} which can be adapted to our context. We start by proving the following Caccioppoli-type estimate.

\begin{lemma}\label{Caccioppoli} Let $u\in\WL$ be a weak solution to $\L u=0$ in $\O$. There exists $C>0$ (independent of $u$) such that, for any compact $K\subset\O$, we have
$$\left\|Xu\right\|_{L^2(K)}\leq\frac{C}{\dist(K,\de\O)}\left\|u\right\|_{L^2(\O)},$$
where $\dist(K,\de\O)$ denotes the $d$-distance between $K$ and $\de\O$.
\end{lemma}
\begin{proof} First we prove that, if $0<\rho<r$,
\begin{equation}\label{caccio}
\left\|Xu\right\|_{L^2(B_\rho)}\leq\frac{C}{r-\rho}\left\|u\right\|_{L^2(B_r)}
\end{equation}
for any ball $B_r$ compactly contained in $\O$ (and thus for any ball in $\O$). Let us take a continuous nonnegative cut-off function $\eta$ satisfying
$$\eta\equiv1\mbox{  in }B_\rho,\quad\eta\equiv0\mbox{  outside }B_r, \mbox{   and }\left|X\eta\right|\leq\frac{c}{r-\rho} \mbox{ a.e.}$$
for some positive constant $c$ (for the existence of such a function see e.g. \cite{FSSC}). By the fact that $\LL(u,\eta^2u)=0$ and condition (\ref{xell}), we get
$$\lambda\int_{B_r}\eta^2\left|Xu\right|^2\leq 2\Lambda\int_{B_r}\eta\left|u\right|\left|Xu\right|\left|X\eta\right|
\leq\frac{\lambda}{2}\int_{B_r}\eta^2\left|Xu\right|^2+\frac{2\Lambda^2}{\lambda}\int_{B_r}u^2\left|X\eta\right|^2.$$
Hence we have
$$\int_{B_\rho}\left|Xu\right|^2\leq\int_{B_r}\eta^2\left|Xu\right|^2\leq\frac{4\Lambda^2}{\lambda^2}\int_{B_r}u^2\left|X\eta\right|^2\leq\frac{4\Lambda^2}{\lambda^2}\frac{c^2}{(r-\rho)^2}\int_{B_r}u^2$$
which proves (\ref{caccio}). We now use a covering argument to conclude the proof. Let $\{B_{r_j}(x_j)\}$ be a countable family of $d$-metric balls such that
$$\O=\bigcup_jB_{r_j}(x_j),\quad r_j=\frac{3}{20}\dist(B_{\frac{r_j}{3}}(x_j),\de\O),\,\,\mbox{and}\quad\sum_{j}\chi_{B_{\frac{4}{3}r_j}(x_j)}\leq C_1\chi_{\O}$$
for some positive constant $C_1$ (see \cite[Lemma 2.15]{FSSC}), where $\chi_E$ denotes the characteristic function of the set $E$. Let us fix a compact set $K\subset\O$ and let $F=\{j\in\mathbb{N}\,:\,B_{r_j}(x_j)\cap K\neq \emptyset\}$. Thus, for any $j\in F$, we have
$$\dist(K,\de\O)\leq 2r_j+\dist(\overline{B_{r_j}(x_j)},\de\O)\leq\frac{26}{3}r_j.$$
Hence, since $B_{\frac{4}{3}r_j}(x_j)\subseteq\O$ and by (\ref{caccio}), we get
$$\left\|Xu\right\|^2_{L^2(K)}\leq\sum_{j\in F}\left\|Xu\right\|^2_{L^2(B_{r_j}(x_j))}\leq\sum_{j\in F}\frac{C^2}{(\frac{4}{3}r_j-r_j)^2}\left\|u\right\|^2_{L^2(B_{\frac{4}{3}r_j}(x_j))}
\leq C_1\frac{(26C)^2}{(\dist(K,\de\O))^2}\left\|u\right\|^2_{L^2(\O)}.$$
\qed\end{proof}

Let us put
$$\B=\{u\in\WL\,:\,\left\|u\right\|_{\B}=\left\|u\right\|_{L^\infty(\O)}+\sup_{K\subset\subset\O}\dist(K,\de\O)\left\|Xu\right\|_{L^2(K)}<+\infty\}.$$
The space ($\B, \left\|\cdot\right\|_{\B}$) is a Banach space, and the inclusion in $\WL$ is continuous. The previous lemma tells us that the linear map $\tilde{B}: H  \longrightarrow \B$ is well-defined and bounded.\\
Let us consider also the space $\mathcal{C}=\frac{\W\cap C(\overline{\O})}{C_0(\overline{\O})}$, where $C_0(\overline{\O})$ is the space of continuous functions in $\overline{\O}$ vanishing on $\de\O$. By Tietze extension theorem, it can be thought as a subspace of ($C(\de\O), \max_{\de\O}\left|\cdot\right|$). Moreover, by Stone-Weierstrass theorem, $\mathcal{C}$ is dense in $C(\de\O)$.\\
By the inclusion map $\W\cap C(\overline{\O})\hookrightarrow\{h\in\W\,:\,\sup_{\de\O}\left|h\right|<+\infty\}$, we have a well-defined map $j:\mathcal{C}\rightarrow H$ with $\left\|j\varphi\right\|_H\leq\max_{\de\O}\left|\varphi\right|$. We define $\BB=\tilde{B}\circ j$. By continuous linear extension, we can extend it to a linear and bounded map $$\BB: C(\de\O)\rightarrow\B.$$
By (\ref{MP}), it is easy to see that $\sup_{\O}\BB\varphi\leq\max_{\de\O}\varphi$ for any $\phi\in C(\de\O)$. Furthermore, by density, the $\WL$-function $\BB\varphi$ is a weak solution to $\L (\BB\varphi)=0$ in $\Omega$ for every $\varphi$. The Harnack inequality in \cite[Theorem 4.1]{GL} ensures the continuity in $\O$ of the functions $\BB\varphi$. Thus, we are finally in the position to give the definition of regular points.

\begin{definition}\label{reg} A point $y\in\de\O$ is said to be ($\L$-)regular if, for any $\varphi\in C(\de\O)$, we have
$$\lim_{\O\ni x\rightarrow y}\BB\varphi(x)=\varphi(y).$$
\end{definition}

Arguing as in \cite[Lemma 3.1]{LSW}, we can prove that it is enough to check the regularity condition just for all the functions $\varphi\in \mathcal{C}$.\\
The notion of regularity is classically related with barrier functions. It seems natural to model our definition of barrier on the maximum principle (\ref{MP}).

\begin{definition}\label{barriera} A function $V_y\in\W$ is called a barrier at $y\in\de\O$ if
\begin{itemize}
\item[(i)] $V_y$ is a solution to $\L V_y=0$ in $\O$;
\item[(ii)] $\lim_{\O\ni x\rightarrow y}V_y(x)=0$;
\item[(iii)] for every $\rho>0$ (small enough) there exists $\psi\in C_0^{\infty}(\RN\smallsetminus B_{\rho}(y))$, $0\leq\psi\leq 1$ with $\psi\equiv 1$ in a neighborhood of $\de\O\smallsetminus B_{2\rho}(y)$ and there exists $m>0$ such that $\psi(m-V_y)^+\in\WO.$
\end{itemize}
\end{definition}

\begin{remark}\label{barpos} We note that condition (ii) in the previous definition makes sense since $V_y$ is continuous by condition (i). Moreover, we stress that a barrier function has to be nonnegative in $\O$. As a matter of fact, for any positive $\e$, there exists $\rho$ such that $V_y(x)\geq -\e$ if $x\in\O$ with $d(x,y)\leq 2\rho$. Fixed such a $\rho$, let us consider $\psi$ and $m$ as in condition (iii). We have
$$0\leq(-\e-V_y)^+=\psi(-\e-V_y)^++(1-\psi)(-\e-V_y)^+\leq\psi(m-V_y)^+\in\WO.$$
Thus, $(-\e-V_y)^+\in\WO$ and $\L(-\e-V_y)=0$. By (\ref{MP}), we get $V_y\geq-\e$ in $\O$ for any $\e>0$.
\end{remark}

\begin{proposition}\label{eqbarreg} A point $y\in\de\O$ is regular if and only if there exists a barrier $V_y$ at $y$.
\end{proposition}
\begin{proof} Suppose first $y$ is regular. Put $\Phi(x)=d(x,y)$. The function $\Phi$ belongs to $\W\cap C(\overline{\O})$ (see e.g. \cite{FSSC}). We denote by $\phi$ its restriction to $\de\O$ and we consider $V_y=\BB\phi$. By definition, since $\phi\in \mathcal{C}$, $V_y\in\W$ (and not just in $\WL$). The function $V_y$ is actually a barrier. Conditions (i) and (ii) follow indeed by construction and regularity assumption. Moreover, if we fix a positive $\rho$ and we consider a $C_0^{\infty}$ nonnegative function $\psi$ which vanishes in $\overline{B_\rho(y)}$, we have $0\leq\psi(\rho-V_y)^+\leq\psi(\rho-\Phi)^++\psi(\Phi-V_y)^+=\psi(\Phi-V_y)^+\in\WO$. Thus, even condition (iii) is satisfied.\\
Suppose now the existence of a barrier $V_y$. For what we said after Definition \ref{reg}, it is enough to check that $\BB\phi(x)\rightarrow\phi(y)$ as $x\rightarrow y$ for any $\phi\in \mathcal{C}$. Let us consider $\Phi\in\W\cap C(\overline{\O})$ whose restriction is a fixed $\phi\in \mathcal{C}$. Put $M=\max_{\overline{\O}}\left|\Phi\right|$. By continuity, for any positive $\e$, there exists $\rho$ such that $\left|\Phi(x)-\Phi(y)\right|\leq\e$ if $x\in\overline{\O}$ with $d(x,y)\leq 2\rho$. Fixed such a $\rho$, since we have a barrier we can consider $\psi$ and $m$ as in condition (iii). We get $0\leq h(x):=(\Phi(x)-\Phi(y)-\e-\frac{2M}{m}V_y(x))^+\leq\psi(2M-\frac{2M}{m}V_y)^+\in\WO$ by using that $V_y\geq 0$. Hence we have
$0\leq(\BB\phi-\Phi(y)-\e-\frac{2M}{m}V_y)^+\leq h+(\BB\phi-\Phi)^+\in\WO$. Since $\L(\BB\phi-\Phi(y)-\e-\frac{2M}{m}V_y)=0$ in $\O$, the maximum principle implies that $\BB\phi\leq\Phi(y)+\e+\frac{2M}{m}V_y$ in $\O$. Considering also $\tilde{h}=(\Phi(y)-\Phi-\e-\frac{2M}{m}V_y)^+$, we get at the end
$$\left|\BB\phi-\Phi(y)\right|\leq\e+\frac{2M}{m}V_y\,\,\,\mbox{ in }\O.$$
The fact that $V_y(x)\rightarrow 0$ as $x\rightarrow y$ and the arbitrariness of $\e$ complete the proof.\qed
\end{proof}

\section{Capacitary potentials and regularity}\label{cp3}

Given a compact set $K\ss D$, we define the ($\L$-)capacity of $K$ in $D$ as
 $$
 \cc(K)=\inf\{\LL(u,u)\,|\,u\in W_0^1(D,X),\,u\ge 1\text{ on }K\text{ in the $W_0^1(D,X)$ sense}\}.
 $$
 We say that $u\ge c$ on $K$ in the $W_0^1(D,X)$ sense, if \tes a sequence $\phi_j\in\LIPO(D)$ (Lipschitz functions compactly supported in $D$) \st
 $\phi_j\ge c$ on $K$ and $\phi_j\to u$ in $W^1(D,X)$. We shall also say that $u=c$ on $K$ in the $W_0^1$ sense if both $u\ge c$ and
 $-u\ge -c$ on $K$ in the $W_0^1$ sense. Because of the $X$-ellipticity condition, it can be showed that there exists a unique $u_0\in W_0^1(D,X)$ with $u\ge 1$ on $K$ in the $W_0^1(D,X)$ sense such that $\cc(K)=\LL(u_0,u_0)$. The function $u_0$ is called the capacitary potential of $K$ in $D$. It is also a weak solution to $\L u_0=0$ in $D\setminus K$. Furthermore, there exists a positive measure $\mu_0$ supported on $\de K$ \st $\LL(u_0,\phi)=\int\phi\,\dd\mu_0$ \fe $\phi\in C_0^\infty(D)$. The measure $\mu_0$ is called the capacitary distribution of $K$ in $D$ and it holds $\mu_0(K)=\LL(u_0,u_0)=\cc(K)$ (see \cite[Section 3]{U}).

\begin{remark}\label{urem}
Let us consider a function $h\in W^1(D,X)\cap C(\overline{D})$ such that $h\equiv 1$ in a neighborhood of $K$ and $h\equiv 0$ in a neighborhood of $\de D$. With our notations and with $\O=D\smallsetminus K$, we have
$$u_0=\BB(h_{|\de\O}).$$
As a matter of fact, by \cite[Proposition 3.1]{U} there exists a sequence $\phi_j\in \rm{Lip_0}(D)$ with $\phi_j\equiv 1$ in $K$ such that $\phi_j\rightarrow u_0$ in $W^1(D,X)$. The functions $\phi_j-h$ vanish on $K$ and in a neighborhood of $\de D$. We thus have $\phi_j-h\in W_0^1(\O,X)$ and hence $u_0\in W_0^1(\O,X)$.
\end{remark}

With these notions, we can show that the existence of our barriers (and so the regularity) at some point $y$ is a local issue. 

\begin{lemma}\label{indue} Let $\O,\O_0$ be open sets compactly contained in $D$, with $\O\subseteq\O_0$. Let $y\in\de\O\cap\de\O_0$ and suppose there exists $\delta>0$ such that $\O\cap B_\delta(y)=\O_0\cap B_\delta(y)$. Then $y$ is regular for $\O$ if and only if it is regular for $\O_0$.
\end{lemma}
\begin{proof} Suppose first $y$ is regular for $\O$. If $\Phi(x)=d(x,y)$, by the proof of Proposition \ref{eqbarreg}, the function $V_y=\BB(\Phi_{|\de\O})$ is a barrier. We want that $V_0=\BB_0(\Phi_{|\de\O_0})$ is a barrier at $y$ for $\O_0$, where $\BB_0$ denotes the operator we built up in the previous section related to $\O_0$. To prove this, we will not exploit that $\O_0$ is compactly contained in $D$. We claim there exists $C>0$ such that $V_0\leq C V_y$ in $\O$: so we get condition (ii) in Definition \ref{barriera}, which is the one missing. Since $V_0$ in $\O_0$ is bounded by construction, we can take $M>\frac{\delta}{2}$ such that $V_0\leq M$. Let us consider $\psi\in C_0^{\infty}(\RN\smallsetminus B_{\frac{\delta}{2}}(y))$, $0\leq\psi\leq 1$, $\psi\equiv 1$ in a neighborhood of $\overline{\O_0}\smallsetminus B_{\delta}(y)$. If $C=\frac{2M}{\delta}$, we have
$$(V_0-CV_y)^+\leq C\psi\left(\frac{\delta}{2}-V_y\right)^++(1-\psi)(V_0-\Phi)^++(1-\psi)(\Phi-V_y)^+.$$
By noting that $V_0-\Phi\in W_0^1(\O_0)$ and $(1-\psi)$ is supported in $B_\delta(y)$, we get that the right hand side of the above inequality belongs to $\WO$ and so does $(V_0-CV_y)^+$. By (\ref{MP}), we deduce the claim and the fact that $y$ is regular even for $\O_0$. \\
Suppose now that $y$ is regular for $\O_0$. Since $\O_0$ is compactly contained in $D$, we can find $D_0\subseteq D$ satisfying condition \eqref{AI} with $\O_0\subset D_0$. Fix $\rho<\delta$ such that $\overline{B_\rho(y)}\subset D_0$. By the first part of the proof, $y$ is regular even for $D_0\smallsetminus \left(\overline{B_{\frac{\rho}{k}}(y)}\smallsetminus\O_0\right)$ (for all positive integers $k$). We denote by $u_k$ the capacitary potentials of $\overline{B_\frac{\rho}{k}(y)}\smallsetminus\O_0$ in $D_0$. We want to prove that $V_y:=\sum_{k=2}^{+\infty}2^{-k}(1-u_k)$ defines a barrier in $\O$. Let us first note that $\overline{B_r(y)}\smallsetminus\O=\overline{B_r(y)}\smallsetminus\O_0$ for every positive $r\leq\rho$, and that $\L V_y=0$ in $\O$. Moreover, the assumptions on $D_0$ assure that $u_k$ are (equal a.e. to) continuous functions on $\overline{D_0}\smallsetminus\left(\overline{B_\frac{\rho}{k}(y)}\smallsetminus\O_0\right)$ (see \cite[Lemma 2.5]{U}), vanishing on $\de D_0$ and bounded by $1$. Hence, for any positive $\sigma$, consider an integer $k_0$ such that $B_{\frac{\rho}{k_0}}(y)\subset B_{\frac{\sigma}{2}}(y)$ and a positive number $m_\sigma$ with $m_\sigma\leq 2^{-k_0}(1-u_{\frac{\rho}{k_0}}(x))$ for a.e. $x\in\overline{D_0}\smallsetminus B_\sigma(y)$. We stress that the existence of such $m_\sigma$ is provided by a strong maximum principle (following from the Harnack inequality proved in \cite{GL}). Thus, for $\psi\in C_0^{\infty}(\RN\smallsetminus B_{\sigma}(y))$, $0\leq\psi\leq 1$ with $\psi\equiv 1$ outside $B_{2\sigma}(y)$, we get $\psi(m_\sigma-V_y)^+\leq\psi(m_\sigma-2^{-k_0}(1-u_{\frac{\rho}{k_0}}))^+=0$. Finally, the regularity of $y$ and Remark \ref{urem} imply $\lim_{\O\ni x\rightarrow y}{u_k(x)}=1$ and $V_y(y)=0$. This proves that $V_y$ is a barrier at $y$ for $\O$.
\qed\end{proof}

The previous lemma allows us to consider, without any loss of generality, a domain $D$ satisfying condition \eqref{AI}, and such that the conditions \eqref{DC}, \eqref{PI}, and \eqref{uconn} hold true for any $d$-ball contained in $D$. For such $D$, we can now introduce and exploit the Green function $g$ related to $D$. The main result in \cite[Theorem 3.9]{U} is that
\begin{equation}\label{TGE}
C^{-1}\int_{d(x,y)}^{\dist_d(x,\de D)}\frac{s}{|B_s(x)|}\,\,\dd s\le g(x,y) \le C\, \int_{d(x,y)}^{\dist_d(x,\de D)}\frac{s}{|B_s(x)|}\,\,\dd s,
\end{equation}
for every  $x,y\in D$ \st $0<d(x,y)\le 10^{-2}\dist_d(x,\de D)$. In particular we deduce that $g(x,y)\rightarrow+\infty$ as $d(x,y)\rightarrow 0$ since $\left|B_s(x)\right|\leq c s^{N}$ by the boundedness of the coefficients of the vector fields $X_j$'s in $\overline{D}$. This fact and the continuity of $g$ outside the diagonal (see \cite[Theorem 3.4]{U}) imply also that the Green kernel $g$ is lower semicontinuous in $D\times D$.

From now on, we fix $\O$ as a bounded open set compactly contained in $D$. For $y\in\de\O$ and $\rho>0$ such that $\overline{B_\rho(y)}\subset D$, we denote $K_\rho=\overline{B_\rho(y)}\smallsetminus\O$ and by $u_\rho$ and $\mu_\rho$ the capacitary potential and distribution of the compact set $K_\rho$. More precisely, $u_\rho$ will denote the lower semi-continuous representative of the capacitary potential of $K_\rho$, i.e.
$$u_\rho(x)=\int g(x,y)\,d\mu_\rho(y)$$
(see \cite[Lemma 3.5]{U}). This function is continuous outside $K_\rho$ and in its interior, where it is respectively less than and equal to $1$. By lower semi-continuity it has to be bounded by $1$ everywhere.

We want to prove the counterpart in our setting of some classical characterizations for the regularity in terms of the behavior of $u_\rho$. To do this, we mainly follow the arguments in \cite[Section 5]{FJK}. We stress that, unlike in \cite{FJK}, we are not going to use fine properties of quasi-continuity and capacity. \\
We start with the following lemma, where we denote by $W^{-1}$ the dual space of $W_0^1(D,X)$.

\begin{lemma}\label{primautile} Let $\mu$ be a positive Borel measure compactly supported in $D$. Suppose also $\mu\in W^{-1}$. Then
$$u(x):=\int_D g(x,y)\,d\mu(y)\,\,\in W_0^1(D,X).$$
\end{lemma}
\begin{proof} By the $X$-ellipticity and Lax-Milgram Theorem, there exists (a unique) $v\in W_0^1(D,X)$ such that $\LL(v,\phi)=\left\langle \mu,\phi\right\rangle_{W^{-1},W_0}$ for any $\phi\in W_0^1(D,X)$. Moreover, if $\phi\in C_0^1(D)$, we have $\LL(v,\phi)=\int\phi d\mu$. By arguing as in the proof in \cite[Lemma 3.5]{U}, we actually have
\begin{equation}\label{utuno}
\LL(v,\phi)=\int\phi\,d\mu\quad \mbox{for every }\phi\in C(\overline{D})\cap W_0^1(D,X)
\end{equation}
and $\int v(x)h(x)\,dx=\int u(x)h(x)\,dx$ for any $h\in L^p$ (for a fixed $p>\frac{Q}{2}$: we remind that $u\in L^{p'}$ by \cite[Theorem 3.4]{U}). Since $v\in W_0^1(D,X)$, we get $u=v$ and the assertion.
\qed\end{proof}

The following proposition is crucial to our aim.

\begin{proposition}\label{superpp} Let $\mu\in W^{-1}$ be a positive Borel measure compactly supported in $D$. Let us consider
$u(x)=\int_D g(x,y)\,d\mu(y)$ (which belongs to $W_0^1(D,X)$ for the last lemma), and suppose $u$ is bounded. Then
$$u(y)\geq\liminf_{x\rightarrow y}\tilde{u}(x)\qquad\mbox{for any function }\tilde{u}=u\mbox{ almost everywhere}.$$
\end{proposition}
\begin{proof} For any fixed $a>0$, we put $F_a(t)=t$, if $t\leq a$, $F_a(t)=t-\frac{1}{4a}(t-a)^2$, if $a\leq t\leq 3a$, and $F_a(t)=2a$ if $t\geq 3a$. We consider $h_a=F_a(g(y,\cdot))$ as in \cite[Section 8]{LSW}. We have $h_a\in W_0^1(D,X)\cap C(\overline{D})$. By monotone convergence we also get $u(y)=\lim_{a\rightarrow+\infty}\int h_a(z)d\mu(z)$. In the distributional sense, $-\L h_a=f_a\in L^1$, where $f_a(t)$ can be thought as $\frac{1}{2a}\left\langle B(x)\nabla g(y,x),\nabla g(y,x)\right\rangle$, if $a\leq g(y,x)\leq 3a$, and it vanishes elsewhere. We note that $f_a\in L^1$ by the $X$-ellipticity and the fact that $g(y,\cdot)$ is in $W^1$ outside any neighborhood of \{$y$\}. Moreover $f_a\geq 0$ and it is supported in a compact $K\subset D$. We claim that
$$\int h_a(z)\,d\mu(z)=\int u(x)f_a(x)\,dx.$$
We know that $\LL(h_a,\phi)=\int\phi(x) f_a(x)dx$ for any $\phi\in C_0^1(D)$. Let $\phi\in L^{\infty}\cap W_0^1(D,X)$. If we take $\psi\in C_0^{\infty}(D)$ with $0\leq\psi\leq 1$ and $\psi\equiv 1$ in a neighborhood of $K$, then we have $(1-\psi)\phi\in W_0^1(D\smallsetminus K,X)$ and $\LL(h_a,(1-\psi)\phi)=0$. On the other hand, the mollifiers $(\psi\phi)_{\frac{1}{n}}$ converge both uniformly and in $W_0^1$ to $\psi\phi$ in $K$ (see \cite[Proposition 1.4]{FSSC}). Since $\LL(h_a,(\psi\phi)_{\frac{1}{n}})=\int (\psi\phi)_{\frac{1}{n}}f_a$, we get $\LL(h_a,\psi\phi)=\int\phi f_a$. Thus $\LL(h_a,\phi)=\int\phi f_a$ for all $\phi\in L^{\infty}\cap W_0^1(D,X)$, and in particular for $\phi=u$. By (\ref{utuno}) and the symmetry of $\LL$ in $W_0^1\times W_0^1$, we get $\int h_a(z)\,d\mu(z)=\LL(u,h)=\LL(h,u)=\int u(x)f_a(x)\,dx$ and the claim is proved. Therefore
$$u(y)=\lim_{a\rightarrow+\infty}\int u(x)f_a(x)\,dx.$$
Now we can follow closely the arguments in \cite{FJK}. Let us put $J_a=\{x\in D\,:\,g(y,x)\geq a\}$. The set $J_a$ is compact and we denote by $v_a$ and $\nu_a$ respectively its capacitary potential and distribution. Since $v_a(x)=\int g(x,z)d\nu_a(z)$ is continuous at $y\in\rm{int}(J_a)$ and $\de J_a\subseteq\{x\in D\,:\,g(y,x)= a\}$, we have $1=v_a(y)=a\cc(J_a)$. We also have $v_a(x)=\frac{1}{a}g(y,x)$ outside $J_a$, since they solve the same Dirichlet problem for $\L$ in $D\smallsetminus J_a$ (see Remark \ref{urem}, and note that $g(y,\cdot)$ is continuous up to the boundary of $D\smallsetminus J_a$). So we get $\frac{1}{a}=\LL(v_a,v_a)=\frac{1}{a^2}\int_{D\smallsetminus J_a}\left\langle B(x)\nabla g(y,x),\nabla g(y,x)\right\rangle$, which implies $\int f_a=1$. Hence, we finally deduce that $u(y)=\lim_{a\rightarrow+\infty}\int \tilde{u}(x)f_a(x)\,dx\geq \liminf_{x\rightarrow y}\tilde{u}(x)$ for any function $\tilde{u}=u$ almost everywhere.
\qed\end{proof}

We are now in the position to state and prove the following characterizations of the regularity of a boundary point.

\begin{proposition} Let $y\in\de\O$. We have
\begin{equation}\label{riffxty}
\textstyle
y\mbox{ is regular} \qquad \mbox{iff}\qquad \lim_{\O\ni x\rightarrow y}{u_\rho(x)}=1\mbox{ for all }\rho>0,
\end{equation}
\begin{equation}\label{unodidue}
y\mbox{ is regular} \qquad \mbox{iff}\qquad u_\rho(y)=1\mbox{ for all }\rho>0,
\end{equation}
\begin{equation}\label{duedidue}
y\mbox{ is not regular} \qquad\mbox{iff}\qquad \lim_{\rho\rightarrow 0^+}u_\rho(y)=0.
\end{equation}
\end{proposition}
\begin{proof} We start with \eqref{riffxty}. Suppose $y$ is regular. For the first part of the proof of Lemma \ref{indue}, $y$ is also regular for $\O_0=D\smallsetminus K_\rho$, for any fixed $\rho>0$. Then, Remark \ref{urem} implies $\lim_{\O\ni x\rightarrow y}{u_\rho(x)}=1$. Viceversa, if we suppose $u_\rho\rightarrow 1$ for every $\rho$, we can use the barrier $V_y:=\sum_{k=2}^{+\infty}2^{-k}(1-u_{\frac{\rho}{k}})$ as in the second part of Lemma \ref{indue}.\\
In order to prove \eqref{unodidue}, we first note that the measure $\mu_\rho$ belongs to $W^{-1}$. In fact we have $\left|\int\phi d\mu_\rho\right|=\left|\LL(u_\rho,\phi)\right|\leq C\left\|u_\rho\right\|_{W_0^1}\left\|\phi\right\|_{W_0^1}$ for any $\phi\in C_0^1(D)$. Moreover $u_\rho$ is bounded by $1$. Hence we can apply Proposition \ref{superpp} with $\tilde{u}=u_\rho$ in $D\smallsetminus K_\rho$ and equal to $1$ in $K_\rho$. We deduce that, for all positive $\rho$, $u_\rho(y)\geq\liminf_{D\smallsetminus K_\rho\ni x\rightarrow y}\tilde{u}(x)=\liminf_{D\smallsetminus K_\rho\ni x\rightarrow y}u_\rho(x)$. On the other hand, the lower semicontinuity provides the opposite inequality. Hence we get
\begin{equation}\label{inf}
u_\rho(y)=\liminf_{\O\ni x\rightarrow y}u_\rho(x).
\end{equation}
Thus, (\ref{unodidue}) follows from (\ref{riffxty}) and the fact that $u_\rho\leq 1$.\\
Let us turn to (\ref{duedidue}). If $\lim_{\rho\rightarrow 0^+}u_\rho(y)=0$, then $y$ is certainly not regular by (\ref{unodidue}). Viceversa, if we suppose that $y$ is not regular, for some $\rho_0$ we have $u_{\rho_0}(y)<1$. We are going to adapt the arguments in \cite[Lemma 5.7]{FJK}. First we recognize that
\begin{equation}\label{limru}
\lim_{r\rightarrow 0^+}\int_{B_r(y)}g(y,z)\,d\mu_\rho(z)=0.
\end{equation}
This holds true by dominated convergence and by the fact that $\cc(\{y\})=0$ (see \cite[Proposition 3.6]{U}) which implies $\mu_\rho(\{y\})=0$. Thus, if we fix $\e>0$, there exists $\sigma<\rho_0$ such that $\int_{\overline{B_\sigma(y)}}g(y,z)\,d\mu_\rho(z)\leq\e$. Put $u_{\rho_0}(x)=\int_{\overline{B_\sigma(y)}}g(x,z)\,d\mu_\rho(z)+\int_{D\smallsetminus\overline{B_\sigma(y)}}g(x,z)\,d\mu_\rho(z)=v(x)+u(x)$. We know that $v(y)\leq\e$ and $v,u\in W_0^1(D,X)$ since we can apply Lemma \ref{primautile}. Moreover, $u(y)\leq u_{\rho_0}(y)<1$ and $u$ is continuous in $y$. Hence, there exists $\tau_0<\frac{\sigma}{2}$ such that $u\leq\frac{1}{2}(1+u_{\rho_0}(y))$ in $B_{2\tau_0}(y)$. We fix $\tau\leq\tau_0$ and we put $h(x)=\frac{1}{2}(1-u_{\rho_0}(y))u_\tau(x)-v(x)\in W_0^1(D,X)$. Let us take a $C_0^{\infty}$-function $\psi$ such that $0\leq\psi\leq 1$, $\psi\equiv 1$ in a neighborhood of $D\smallsetminus B_{2\tau}(y)$, and $\psi\equiv 0$ in a neighoborhood of $K_\tau$. Therefore $\psi h^+\in W_0^1(D\smallsetminus K_\tau,X)$ and
$$(1-\psi)h^+\leq (1-\psi)\left(\frac{1}{2}(1-u_{\rho_0}(y))-u_{\rho_0}+\frac{1}{2}(1+u_{\rho_0}(y))\right)^+
=(1-\psi)(1-u_{\rho_0})^+\in W_0^1(D\smallsetminus K_\tau,X)$$
since $u_{\rho_0}$ can be approximated in $W^1$-norm by a sequence of functions which are identically $1$ in $K_{\rho_0}\supseteq K_\tau$. Thus, $h^+\in W_0^1(D\smallsetminus K_\tau,X)$. On the other hand
$$\LL(h,\phi)=-\int_{\overline{B_\sigma(y)}}\phi\,d\mu_\rho\leq 0\qquad \mbox{for any }\phi\in C_0^1(D\smallsetminus K_\tau),\,\phi\geq 0,$$
i.e. $h$ is a weak subsolution of $\L u=0$. By the maximum principle proved in \cite[Theorem 3.1]{GL}, we get $\sup_{D\smallsetminus K_\tau} h^+\leq \sup_{\de(D\smallsetminus K_\tau)}h^+=0$ that is $v\geq\frac{1}{2}(1-u_{\rho_0}(y))u_\tau$ almost everywhere in $D\smallsetminus K_\tau$. But, in $D\smallsetminus K_\sigma$, $v$ and $u_\tau$ are continuous and so $v\geq\frac{1}{2}(1-u_{\rho_0}(y))u_\tau$ always in $D\smallsetminus K_\sigma$. Hence, by (\ref{inf}), we have $\liminf_{\O\ni x\rightarrow y}v(x)\geq\liminf_{\O\ni x\rightarrow y}\frac{1}{2}(1-u_{\rho_0}(y))u_\tau(x)=\frac{1}{2}(1-u_{\rho_0}(y))u_\tau(y)$. On the other hand, if we set $E=\{x \in K_\tau\,:\,v\geq\frac{1}{2}(1-u_{\rho_0}(y)) \}$, we have $\left|K_\tau\smallsetminus E\right|=0$. By considering $\tilde{v}=v$ in $(D\smallsetminus K_\tau)\cup E$ and equal to $1$ elsewhere, we can apply Proposition \ref{superpp} and deduce that $\e\geq v(y)\geq\liminf_{x\rightarrow y}\tilde{v}(x)=\liminf_{D\smallsetminus K_\tau\ni x\rightarrow y}v(x)$ since $v(x)>\e\geq v(y)$ for $x\in E$ (if $\e<\frac{1}{2}(1-u_{\rho_0}(y))$). In conclusion we get $\e\geq\frac{1}{2}(1-u_{\rho_0}(y))u_\tau(y)$ for any $\tau\leq\tau_0$ which proves the desired implication.
\qed\end{proof}

The characterization given by \eqref{duedidue} will be crucial also to get the regularity for boundary points with an exterior cone-type property. We will prove this fact at the end of the next section.

\section{Wiener's integral and regularity}\label{www}

In this section we finally prove our Wiener criterion. In the following lemma we try to extrapolate the essential tools in order to avoid the quasi-continuity issue for the capacitary potentials.

\begin{lemma}\label{mezzo} Let $0<r<\rho$. Then, in our notations we have
\begin{itemize}
\item[(i)] $\mu_\rho(K_r)\leq\cc(K_r)$\qquad and
\item[(ii)] there exists a positive constant $C$ (independent of $r, \rho$) such that
$$\cc(K_{\frac{r}{4}})\leq\mu_\rho(K_r)+C\cc(K_{\frac{r}{4}})u_\rho(y).$$
\end{itemize}
\end{lemma}
\begin{proof} Let $\phi_j^\rho$ and $\phi_j^r$ be two sequence of functions in $\rm{Lip}_0(D)$ approaching respectively $u_\rho$ and $u_r$ in the $W^1(D,X)$-norm and such that $\phi_j^\rho\equiv 1$ in $K_\rho$, $\phi_j^r\equiv 1$ in $K_r$. We have $\LL(u_\rho,u_r)=\LL(u_r,u_\rho)$. On one side we get 
$$\LL(u_r,u_\rho)=\lim_{j\rightarrow+\infty}\LL(u_r,\phi_j^\rho)=\lim_{j\rightarrow+\infty}\int_{K_r}\phi_j^\rho\,d\mu_r=\mu_r(K_r)=\cc(K_r)$$
since $K_\rho\supset K_r$. On the other hand
$$\LL(u_\rho,u_r)=\lim_{j\rightarrow+\infty}\LL(u_\rho,\phi_j^r)=\lim_{j\rightarrow+\infty}\int_{K_\rho}\phi_j^r\,d\mu_\rho\geq\lim_{j\rightarrow+\infty}\int_{K_r}\phi_j^r\,d\mu_\rho=\mu_\rho(K_r)$$
since we can assume $\phi_j^r\geq 0$, and $(i)$ is proved.\\
In order to prove $(ii)$, we first show that
\begin{equation}\label{firb}
\cc(K_{\frac{r}{4}})\leq\mu_\rho(K_r)+\int_{K_\rho\smallsetminus K_r}u_{\frac{r}{4}}\,d\mu_\rho.
\end{equation}
We take $\psi\in C_0^\infty(D)$ with $0\leq\psi\leq 1$ and $\psi\equiv 1$ in a compact neighborhood $K$ of $K_\rho$. Then $\LL(u_\rho,u_r)=\LL(u_\rho,\psi u_r)$. Let us consider the mollifiers $h_j=(\psi u_\frac{r}{4})_{\frac{1}{j}}$. Hence we get
\begin{eqnarray*}
\LL(u_\rho,u_{\frac{r}{4}})&=&\lim_{j\rightarrow +\infty}\LL(u_\rho,h_j)=\lim_{j\rightarrow +\infty}\int_{K_\rho\smallsetminus K_r}h_j d\mu_\rho + \int_{K_r}h_j d\mu_\rho \\
&\leq& \mu_\rho(K_r) + \lim_{j\rightarrow +\infty}\int_{K_\rho\smallsetminus K_r}h_j d\mu_\rho = \mu_\rho(K_r) + \int_{K_\rho\smallsetminus K_r}u_{\frac{r}{4}} d\mu_\rho
\end{eqnarray*}
since $h_j\rightarrow \psi u_\frac{r}{4}$ in $W^1(K,X)$, $h_j\leq 1$, and $\psi u_\frac{r}{4}$ is continuous in a neighborhood of $K_\rho\smallsetminus K_r$. As in the first part, we have also $\LL(u_{\frac{r}{4}},u_\rho)=\cc(K_\frac{r}{4})$ and (\ref{firb}) is proved. Now we note that, if $x\in K_\rho\smallsetminus K_r$, the $\L$-harmonicity of the Green function and the Harnack inequality in \cite{GL} provide that $g(x,\xi)\leq Cg(x,y)$ for any $\xi\in K_\frac{r}{4}$. Thus we get 
$$\int_{K_\rho\smallsetminus K_r}u_{\frac{r}{4}}\,d\mu_\rho(x)=\int_{K_\rho\smallsetminus K_r}\int_{K_{\frac{r}{4}}}g(x,\xi)\,d\mu_{\frac{r}{4}}(\xi)\,d\mu_\rho(x)\leq C\cc(K_{\frac{r}{4}})u_\rho(y)$$
and the proof is complete.
\qed\end{proof}

Let us now give the full statement of our main result and conclude the proof.

\begin{theorem}\label{thwin} A point $y\in\de\O$ is $\L$-regular if and only if
$$\int_0^{\dist(y,\de D)}\cc(K_\rho)\frac{\rho}{\left|B_\rho(y)\right|}\,d\rho=+\infty.$$ 
\end{theorem}
\begin{proof} Put $R=\dist(y,\de D)$. By (\ref{duedidue}), the statement is equivalent to the following
$$\lim_{\rho\rightarrow 0^+}u_\rho(y)=0\qquad \mbox{iff }\qquad \int_0^{R}\cc(K_\rho)\frac{\rho}{\left|B_\rho(y)\right|}\,d\rho <+\infty.$$
For small $\rho$, the estimates on the Green function (\ref{TGE}) show that $u_\rho(y)$ behaves like $$\int_{K_\rho}\int_{d(x,y)}^R\frac{s}{\left|B_s(y)\right|}\,ds\,d\mu_\rho(x).$$ 
This quantity is, up to constants, equivalent to $$\sum_{j=0}^{+\infty}\int_{2^{-j}\rho}^R\frac{s}{\left|B_s(y)\right|}\,ds\left(\mu_\rho(K_{\frac{\rho}{2^{j}}})-\mu_\rho(K_{\frac{\rho}{2^{j+1}}})\right).$$ The summation by parts and (\ref{limru}) imply then
\begin{equation}\label{tutto}
u_\rho(y)\mbox{ is equivalent to }\cc(K_\rho)\int_\rho^R\frac{s}{\left|B_s(y)\right|}\,ds+\int_0^\rho\frac{s}{\left|B_s(y)\right|}\mu_\rho(K_s)\,ds.
\end{equation}
Suppose first that $M=\int_0^{R}\cc(K_s)\frac{s}{\left|B_s(y)\right|}\,ds <+\infty$. For any positive $\e$ there exists $\delta_0>0$ such that $\int_0^{\delta_0}\cc(K_s)\frac{s}{\left|B_s(y)\right|}\,ds\leq\e$. Moreover, since $s\mapsto \cc(K_s)$ is increasing and tends to $0$ as $s\rightarrow 0^+$ (see \cite[Proposition 3.6]{U}), there exists $0<\delta<\delta_0$ such that $\cc(K_s)\leq\e\cc(K_{\delta_0})$ for every $s <\delta$. By $(i)$ in Lemma \ref{mezzo} and (\ref{tutto}) we get, for $\rho<\delta$, that $u_\rho(y)$ is controlled from above by  $$\int_\rho^{\delta_0}\cc(K_s)\frac{s}{\left|B_s(y)\right|}\,ds+\e\int_{\delta_0}^{R}\cc(K_s)\frac{s}{\left|B_s(y)\right|}\,ds+\e\leq2\e+\e M.$$
Therefore $u_\rho(y)\rightarrow 0$.\\
Viceversa, suppose $\lim_{\rho\rightarrow 0^+}u_\rho(y)=0$. Take $\rho_0>0$ such that $u_\rho(y)\leq \frac{1}{2C}$ for $\rho\leq\rho_0$, where $C$ is the constant appearing in $(ii)$ of Lemma \ref{mezzo}. That lemma infers that, for any $r<\rho<\rho_0$, we have $\cc(K_{\frac{r}{4}})\leq 2\mu_\rho(K_r)$. By the assumption and (\ref{tutto}), we have also $\int_0^\rho\frac{s}{\left|B_s(y)\right|}\mu_\rho(K_s)\,ds<+\infty$. Therefore, for $\rho<\rho_0$, by the doubling property we get
$$\int_0^{\frac{\rho}{4}}\cc(K_s)\frac{s}{\left|B_s(y)\right|}\,ds\leq 2\int_0^{\frac{\rho}{4}}\mu_\rho(K_{4s})\frac{s}{\left|B_s(y)\right|}\,ds\leq 	\frac{A^2}{8}\int_0^{\rho}\mu_\rho(K_s)\frac{s}{\left|B_s(y)\right|}\,ds$$
which is finite. Hence also $\int_0^{R}\cc(K_s)\frac{s}{\left|B_s(y)\right|}\,ds <+\infty$. 
\qed\end{proof}

We finally observe that the $X$-ellipticity condition and the definition of capacity imply that $\cc(K_\rho)$ is in fact equivalent to
$$
 \inf\left\{\left\|Xu\right\|_2^2\,|\,u\in W_0^1(D,X),\,u\ge 1\text{ on }K_\rho\text{ in the $W_0^1(D,X)$ sense}\right\}.
$$
This quantity depends just on the vector fields $X_1,\ldots,X_m$. That is why the following corollary can be directly deduced from our Wiener criterion.

\begin{corollary}\label{corind} For a fixed system of vector fields $X$, any operator in the class of the $X$-elliptic operators defined by (\ref{nomeoper})-(\ref{xell}) has the same regular points for $\O$. In particular the $\L$-regularity of the boundary points of $\O$ does not depend on the coefficients $b_{ij}$ of $\L$ (but only on the vector fields $X$).
\end{corollary}

We conclude with a geometric criterion for the regularity of a boundary point.

\begin{proposition}\label{cone} Let $\Omega$ be an open set compactly contained in $D$ and let $y\in\de\O$. Suppose that there exist $r_1,\,\theta>0$ such that 
$$ |B_r(y)\setminus \O|=\left|K_r\right|\ge\theta|B_r(y)|\qquad\mbox{for any }r\in]0,r_1].$$
Then $y$ is $\L$-regular for any $X$-elliptic operator $\L$.
\end{proposition}
\begin{proof} Let us fix $r_0>0$ such that $\overline{B_{4040r_0}(y)}\subset D$. For any $0<r<r_0$, there exists $D_r$ satisfying condition \eqref{AI} with $B_{303r}(y)\subset D_r \subset B_{404r}(y)$ (by \cite[Lemma 3.8]{U}). We denote by $g_r$ the Green function of $D_r$ and we consider the function
$$v_r(x)=\frac{1}{r^2}\int_{K_r}{g_r(x,z)\,dz}.$$
We note that for any $x\in B_{2r}(y)$ we have $K_r\subset B_{3r}(x)$, and $d(x,z)<3r<10^{-2}\rm{dist}(x,\de D_r)$ for every $x\in B_{2r}(y)$ and $z\in B_{3r}(x).$ By \cite[Theorem 3.9]{U}, $g_r$ satisfies the bounds in \eqref{TGE} for some $C$ independent of $r$. In order to obtain a bound for $v_r$, we are going to use the following inequalities
$$\frac{\left|B_r\right|}{\left|B_s\right|}\geq \frac{1}{A}\frac{r^Q}{s^Q},\quad \rm{and } \quad\frac{\left|B_r\right|}{\left|B_s\right|}\leq \frac{1}{\beta}\frac{r^\mu}{s^\mu},\quad \rm{for }\quad0<r\leq s.$$
The first inequalities comes from the doubling condition \eqref{DC}, whereas the second one follows from the reverse doubling $\left|B_\rho\right|<\beta\left|B_{2\rho}\right|$ which locally holds true in our Carnot-Carath\'eodory setting (see e.g. \cite[Section 2]{DGL}). Here $\mu=\log_2{\frac{1}{\beta}}$ can be assumed less than $2$, since $\beta<1$ can be always thought close to $1$. Thus, for any $x\in B_{2r}(y)$, we get
\begin{eqnarray*}
v_r(x) &\leq& \frac{C}{r^2}\int_{B_{3r}(x)}\int_{d(x,z)}^{\rm{dist}(x,\de D_r)}\frac{s}{\left|B_s(x)\right|}ds\,dz = \frac{C}{r^2}\left(\int_{0}^{3r}s\,ds + \int_{3r}^{\rm{dist}(x,\de D_r)}\frac{s\left|B_{3r}(x)\right|}{\left|B_s(x)\right|}ds\right)\\
&\leq&\frac{C}{r^2}\left(\frac{9}{2}r^2+\frac{(3r)^\mu}{\beta}\int_{3r}^{406r}s^{1-\mu}ds\right)=C\left(\frac{9}{2}+\frac{(3)^\mu}{\beta(2-\mu)}\left((406)^{2-\mu}-3^{2-\mu}\right)\right)=:C_1.
\end{eqnarray*}
On the other hand, by exploiting the cone-type condition and assuming $r\leq r_1$, we deduce
\begin{eqnarray*}
v_r(y)&=&\frac{1}{r^2}\int_{K_r}{g_r(y,z)\,dz}\geq\frac{C}{r^2}\int_{K_r}\int_{d(y,z)}^{\rm{dist}(y,\de D_r)}\frac{s}{\left|B_s(y)\right|}ds\,dz \\
&=& \frac{C}{r^2}\left(\int_{0}^{r}\frac{s\left|K_s\right|}{\left|B_s(y)\right|}ds + \int_{r}^{\rm{dist}(y,\de D_r)}\frac{s\left|K_r\right|}{\left|B_s(y)\right|}ds\right)\geq\frac{C\theta}{r^2}\left(\frac{r^2}{2}+\frac{r^Q}{A}\int_{r}^{\rm{dist}(y,\de D_r)}s^{1-Q}ds\right) \\
&=& C\theta\left(\frac{1}{2}+\frac{1}{A(Q-2)}\left(1-(303)^{2-Q}\right)\right)=:C_2.
\end{eqnarray*}
Hence, we can consider the function $\tilde{v}_r=\frac{v_r}{C_1}\in W_0^1(D_r,X)$ and compare it to the capacitary potential $u_r$ of $K_r$. It is not difficult to see that
$$(\tilde{v}_r-u_r)^+ \in W_0^1(D_r\smallsetminus K_r, X).$$
As a matter of fact, we can take $\psi\in C_0^\infty$, with $\psi\equiv 1$ in an open neighborhood of $K_r$, $0\leq\psi\leq 1$, and supported in $B_{2r}(y)$. Since $(1-\psi)(\tilde{v}_r-u_r)^+\leq (1-\psi)\tilde{v}_r$ and $\psi(\tilde{v}_r-u_r)^+\leq \psi(1-u_r)$, both functions belong to $W_0^1(D_r\smallsetminus K_r, X)$.\\
By the $\L$-harmonicity, we get $\tilde{v}_r\leq u_r$ in $D_r\smallsetminus K_r$ from the maximum principle \eqref{MP} (and from the continuity outside $K_r$). Therefore
$$u_r(y)=\liminf_{\O\ni x\rightarrow y}u_r(x)\geq\liminf_{\O\ni x\rightarrow y}\tilde{v}_r(x)\geq\tilde{v}_r(y),$$
where the last inequality follows from the lower-semicontinuity of $\tilde{v}_r$ and the first equality is our relation \eqref{inf}. We have thus obtained that $u_r(y)\geq \frac{C_2}{C_1}$ for any small positive $r$. The characterization \eqref{duedidue} gives the regularity of $y\in\de\O$ and concludes the proof.
\qed\end{proof}


\end{document}